\documentclass[11pt]{amsart}
\usepackage[utf8]{inputenc}
\usepackage{amsmath,amssymb,amsthm}
\usepackage[dvipsnames]{xcolor}
\usepackage{enumitem}
\usepackage{wrapfig}
\usepackage{float}
\usepackage{thm-restate}
\usepackage{tikz}
\usetikzlibrary{calc}

\usepackage{comment}
\usepackage[top=1in, bottom=1in, left=1in, right=1in, marginpar=1in]{geometry}
\usepackage{graphicx,import}
\usepackage{tikz, tikz-cd}

\usepackage[colorlinks,citecolor=blue!70!black,linkcolor=red!60!black]{hyperref}
\usepackage[nameinlink]{cleveref}

\theoremstyle{plain}
\newtheorem{THM}{Theorem}[section]

\newtheorem{PROP}[THM]{Proposition}
\newtheorem{LEM}[THM]{Lemma}

\newtheorem{COR}[THM]{Corollary}

\theoremstyle{definition}
\newtheorem{DEF}[THM]{Definition}
\newtheorem{RMK}[THM]{Remark}
\newtheorem{EX}[THM]{Example}

\newtheorem{OBS}[THM]{Observation}

\newlist{PROPenum}{enumerate}{1} 
\setlist[PROPenum]{label=(\roman*),ref=\thePROP\,(\roman*)}
\crefname{PROPenumi}{Proposition}{Propositions}

\DeclareMathOperator{\id}{id}

\DeclareMathOperator{\Out}{Out}

\DeclareMathOperator{\PHEquiv}{\thicksim_\text{PHE}}
\DeclareMathOperator{\Homeo}{Homeo}

\DeclareMathOperator{\rad}{rad}

\DeclareMathOperator{\rk}{rk}

\newcommand{\N}{\mathbb{N}}

\newcommand{\R}{\mathbb{R}}

\newcommand{\cG}{\mathcal{G}}

\newcommand{\cV}{\mathcal{V}}
\newcommand{\cU}{\mathcal{U}}
\newcommand{\cP}{\mathcal{P}}

\newcommand{\cS}{\mathcal{S}}

\newcommand{\MG}{\mathcal{MG}}

\def\G{{\Gamma}}

\def\e{{\epsilon}}
\newcommand{\inv}{^{-1}}
\newcommand{\arr}{\rightarrow}
\newcommand{\defeq}{:=}

\newcommand{\isom}{\cong}
\newcommand{\del}{\partial}

\title{The Complexity of Proper Homotopy Equivalence of Graphs}
\author{Hannah Hoganson, Jenna Zomback}
\date{\today}

\begin{document}

\begin{abstract}
    We demonstrate that the proper homotopy equivalence relation for locally finite graphs is Borel complete.
    Furthermore, among the infinite graphs, there is a comeager equivalence class.
    As corollaries, we obtain the analogous results for the homeomorphism relation of noncompact surfaces with pants decompositions.
\end{abstract}

\maketitle

\section{Introduction}

The mapping class group of a surface is the group of homeomorphisms of that surface considered up to homotopy. The group of outer automorphisms of a free group can be realized as the group of homotopy equivalences of a graph whose fundamental group is that free group. These two groups of symmetries are closely related and classical objects of study in geometric group theory. Traditionally, the study of these objects focused on \emph{finite-type} surfaces, those with finitely generated fundamental groups, and finite graphs. The corresponding groups of symmetries are then finitely generated groups. Recently, there has been a considerable rise in interest in \emph{infinite-type} surfaces, those with infinitely-generated fundamental group, and their mapping class groups. In parallel, Algom-Kfir and Bestvina introduced the mapping class group of an infinite graph, a geometric infinite-type analog of $\Out(F_n)$, in \cite{AB2021}. 

The mapping class group of a graph is defined to be the group of proper homotopy equivalences, up to proper homotopy. Thus, graphs which are themselves proper homotopy equivalent have isomorphic mapping class groups, just as homeomorphic surfaces have isomorphic mapping class groups.  So, it is natural to ask the parallel questions: \begin{enumerate}
    \item ``How hard is it to recognize when two surfaces are homeomorphic?" and
    \item ``How hard is it to recognize when two graphs are proper homotopy equivalent?" 
\end{enumerate}

To make the above questions precise, we use the terminology of complexity of equivalence relations from descriptive set theory. 
In general, given an equivalence relation $E$ on a standard Borel space $X$, we say that $E$ is at most as complicated as the equivalence relation $F$ on standard Borel space $Y$, and we write $E \leq_B F$, if there is a \textit{Borel reduction} $\phi: X \to Y$ in the sense that 
\[ 
x \; E \; y \iff \phi(x) \; F \; \phi(y).
\]

Intuitively, existence of such a reduction says that to check if $x$ and $y$ are $E$-related, we can pass the points to $Y$ via $\phi$ and check if they are $F$-related.
Thus, $E$ cannot be any ``harder" of a classification problem than $F$.
The condition that the reduction must be Borel is to ensure that the reduction is constructive. Otherwise, assuming that there are at least as many $F$-classes as there are $E$-classes, we could always use the axiom of choice to map each $E$-class to a unique $F$-class.
If we have $E \leq_B F$ and $F \leq_B E$, we say that $E$ and $F$ are \textit{bireducible} and we write $E \equiv_B F$.
Studying the relative complexities of equivalence relations up to bireducibility, has been a major recent focus of descriptive set theory.

\begin{EX}\label{notable_equivalence_rels}
The following are some important points in the hierarchy of analytic equivalence relations in increasing order of complexity.
    \begin{itemize}
    \item  All orbit equivalence relations of countable groups reduce to the shift action of $\mathbb{F}_2$, the free group on two generators, on $2^{\mathbb{F}_2}$. This is due to Dougherty, Jackson, and Kechris in \cite{DJK94}. 
    \item All orbit equivalence relations of non-archimedean group actions reduce to a universal $S_\infty$ action called $C_\infty$. This is due to Becker and Kechris in \cite{BK96}. This is the complexity of isomorphism of countable connected graphs.
    \item All orbit equivalence relations of Polish group actions reduce to the homeomorphism relation of compact metric spaces. This is due to Zielinski in \cite{Z16}.
    \item All analytic equivalence relations reduce to the isomorphism relation of separable Banach spaces. This is due to Ferenczi, Louveau, and Rosendal in \cite{FLR09}.
\end{itemize}
\end{EX}

We first realize the set of locally finite graphs as a standard Borel space in two different ways, in \Cref{ssec:space_of_graphs}, which makes proper homotopy equivalence an analytic equivalence relation.
Our main result is the following.

\begin{THM}\label{thm:mainIntro}
    Proper homotopy equivalence on the space of locally finite graphs is bireducible with $C_\infty$.  
\end{THM}

We get the same complexity result when considering subspaces of graphs of uniform or bounded degree. 
Our proof relies heavily on the result of Camerlo and Gao in \cite{CG01} that the equivalence relation of homeomorphism on closed subsets of the Cantor set is bireducible with $C_{\infty}$.
The lower bound in our complexity result follows a similar line of reasoning as a response on MathOverflow by Conley, where he outlines a reduction of homeomorphism of open sets in the plane to $C_{\infty}$ using the result of Camerlo and Gao \cite{Con}. 

Other equivalence relations at this level include isomorphism of countable graphs (\cite{H93}) and isomorphism of countable linear orderings (\cite{FS89}). 
These are each examples of isomorphisms of countable structures, which Borel reduce to $C_{\infty}$. 
It should be noted that while being bireducible with $C_{\infty}$ is often referred to as being ``Borel complete," the equivalence relations we are considering are not Borel equivalence relations.

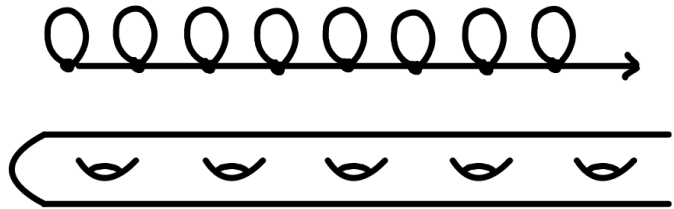
\begin{figure}[h]
    \centering

    \begin{tikzpicture}
   
    \foreach \i in {0,...,8} {
        \node[circle, draw, inner sep=1.5pt, fill=black] (v\i) at (\i,0) {};
        \draw[thick] (v\i) to[out=45,in=135,looseness=20] (v\i);
    }

    \draw[thick] (v0) -- (v8);

    \draw[->, thick] (v8) -- (9,0);
\end{tikzpicture}
\vspace{.6cm}

\begin{tikzpicture}
    \draw[very thick] (0,0) -- (9,0);
    \draw[very thick] (0,-1.2) -- (9,-1.2);
    \draw[very thick] (0,0) arc[start angle=90,end angle=270,radius=0.6];
    \node at (9.2,-.6) {\dots};
    
    \foreach \x in {1,3,5,7} {
        \draw[very thick] (\x,-0.4) to[out=270,in=270,looseness=1.4] ++(1,0);
        \draw[very thick] (\x+.15,-0.7) to[out=70,in=110,looseness=1.4] ++(.7,0);
    }
\end{tikzpicture}
\caption{Representations of the Loch Ness monster graph (top) and surface (bottom).}
    \label{fig:LNM}
\end{figure}

The graph formed by taking a ray and attaching infinitely many loops is often called the \emph{Loch Ness monster graph}, see \cref{fig:LNM}. This was after the corresponding surface, constructed from the plane by attaching handles along the positive real axis was named the \emph{Loch Ness monster surface}. Our second result shows that the equivalence classes of this graph is generic among infinite graphs.  

\begin{THM}\label{thm:GenericIntro}
     The proper homotopy equivalence class of the Loch Ness monster graph is comeager in the space of locally finite, infinite graphs.
\end{THM}

We note that while the PHE equivalence class of the Loch Ness monster graph is comeager, this does \textit{not} come from a comeager isomorphism class.
In fact, in the spaces of graphs considered, every isomorphism class is countable.

In order to realize orientable surfaces as a standard Borel space, we endow them with a pants decomposition. Here a pants decomposition can mean a collection of simple closed curves on the surface up to isotopy so that the complementary components are $d$-holed spheres, for fixed, bounded, or finite $d$. See \Cref{ssec:surfaces} for precise definitions with examples.

\begin{COR}\label{cor:mainIntro}
    The equivalence relation of homeomorphism on the space of surfaces with a pants decomposition is bireducible with $C_\infty$.
\end{COR}

In upcoming work, Bergfalk and Smythe establish another standard Borel space of surfaces and show that the homeomorphism relation is also bireducible with $C_\infty$, \cite{BS}.

The analogous corollary to \Cref{thm:GenericIntro} for surfaces says that the homeomorphism class of the Loch Ness monster surface is a comeager set in the space of noncompact surfaces with a pants decomposition. We interpret this result as saying the following: If you randomly glue pairs of pants together to build a noncompact surface, that surface will almost surely be homeomorphic to a Loch Ness monster surface. This holds true if you allow your building blocks to include annuli or pants with arbitrarily many legs. 

\subsection*{Outline}
In \Cref{sec:prelims}, we give some background information about the spaces we are working in and the tools we use.
In \Cref{sec:results_graphs} we first prove the aforementioned results (\Cref{thm:mainIntro} and \Cref{cor:mainIntro}) for 3-regular graphs.
We then extend the results to other spaces of graphs, including bounded-degree graphs and locally finite graphs.
Finally, we show how to extend these proofs to spaces of surfaces in \Cref{sec:results_surfaces}.

\section*{Acknowledgements} 
The authors would like to thank Denis Osin for suggesting this project along with a starting plan of attack.
Many thanks to Christian Rosendal, Forte Shinko, and Iian Smythe for helpful conversations.
Much appreciation to Lucy Kristoffersen for providing a new perspective. 
The first author was gratefully supported by NSF DMS--2303365 (MSPRF).

\section{Preliminaries}\label{sec:prelims}

\subsection{Graphs}\label{ssec:Graphs}
A \textbf{graph} $\Gamma$ is a triple consisting of two sets: the vertex set and the edge set, and an incidence relation between them associating to each edge two (not necessarily distinct) vertices. We call an edge which is incident to the same vertex twice a \emph{self-loop}. This definition of a graph also allows for \emph{multiple edges}, that is, edges which are incident to the same pair of vertices. Two distinct vertices are \emph{adjacent} if they are incident to the same edge. An \emph{infinite graph} is one where the vertex set or edge set is an infinite set. 

The \textbf{degree} of a vertex, denoted $\deg(v)$, is the number of edges that are incident to it, counting multiplicity, so that a self-loop contributes degree two. We say a graph is \textbf{locally finite} if each vertex has finite degree. A graph is \textbf{$k$-regular} if every vertex in the graph has degree $k$. Note that locally-finite, infinite graphs necessarily have infinite vertex sets.

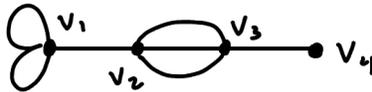
\begin{figure}[h]
    \centering
\begin{tikzpicture}
  \tikzstyle{vertex}=[circle, draw, fill=black, inner sep=0pt, minimum size=4pt]

  \node[vertex, label=above right:$v_1$] (v1) at (0,0) {};
  \node[vertex, label=below:$v_2$] (v2) at (2,0) {};
  \node[vertex, label=above:$v_3$] (v3) at (4,0) {};
  \node[vertex, label=right:$v_4$] (v4) at (6,0) {};

  \draw (v1) -- (v2);
  \draw (v3) -- (v4);

  \draw (v2) -- (v3);
  \draw[bend left=30] (v2) to (v3);
  \draw[bend right=30] (v2) to (v3);

  \draw (v1) to[out=95,in=175,looseness=25] (v1);
  \draw (v1) to[out=185,in=265,looseness=25] (v1);

\end{tikzpicture}
\caption{An example of a graph where $v_1$ has two self loops, and there are multiple edges between $v_2$ and $v_3$. The degree of $v_1$ is $5$, the degrees of $v_2$ and $v_3$ are $4$, and the degree of $v_4$ is $1$.}
    \label{fig:GraphEx}
\end{figure}

We realize a graph as a topological space by building the CW-complex with $0$-cells corresponding to the vertex set of $\Gamma$ and $1$-cells corresponding to the edge set of $\Gamma$. The $1$-cells are then glued to the $0$-cells by the incidence relation. The graph has a natural metric, called the \emph{path metric}, which assigns each edge length one, so that adjacent vertices are distance one from each other. A graph is \emph{connected} if it is connected as a CW complex; for the rest of the paper we will assume all graphs are connected. Note that connected, locally finite, infinite graphs necessarily have infinite but countable vertex and edge sets.  

The fundamental group of a graph is always a free group, so we define the \textbf{rank} of a graph to be $\rk(\Gamma):=\rk(\pi_1(\Gamma))$. Recall that a continuous function between topological spaces, $f:X \arr Y$, is a \textbf{homotopy equivalence} if there is a continuous $g:Y \arr X$ such that $fg$ is homotopic to the identity map, $\id_Y$, and $gf\simeq \id_X$. A function between topological spaces is \textbf{proper} if the preimage of compact sets is compact. Thus, a \textbf{proper homotopy equivalence} is a continuous map $f:X \arr Y$ which is proper with proper homotopy inverse $g:Y \arr X$, so that both $f\circ g$ and $g\circ f$ are properly homotopic to identity maps. It is an exercise in algebraic topology to show that two graphs are homotopy equivalent exactly when they have the same rank. Because we will focus on infinite graphs, we will consider the equivalence relation of proper homotopy equivalence (PHE), and we will write $\Gamma \PHEquiv \Delta$ if $\Gamma$ and $\Delta$ are PHE.

\subsection{End spaces of graphs}\label{ssec:EndSpace}
The classification of locally finite, infinite graphs up to proper homotopy equivalence was done by Ayala, Dominguez, M\'{a}rquez, and Quintero in \cite{ayala1990proper}. In order to understand the classification, we need to define the end space of a graph. 

\begin{DEF}\label{def:Endspace}
    Let $\Gamma$ be a locally finite graph. The \textbf{end space} of $\G$ is the following inverse limit with the inverse limit topology,
  \[ E(\Gamma):= \varprojlim_{K\subset \Gamma} \pi_{0}(\Gamma \setminus K), \] where $\pi_0$ denotes the zeroth homotopy set and $K$ ranges over compact subsets of $\G$. 
\end{DEF}

Because we consider locally finite graphs, they are proper and $\sigma$-compact. That is, closed metric balls are compact, and a countable collection of compact sets cover the graphs. Specifically, for any $\Gamma \in \cG_{\star}$ the collection of closed integer radius metric balls about the basepoint, $\{\overline{B_{n}(\G)}\}_{n\in \N}$,
covers $\Gamma$ and can be used to obtain $E(\Gamma)$. Call these closed balls $B_n$. The transition map $\rho_n: \pi_0(\Gamma \setminus B_{n+1}) \twoheadrightarrow \pi_0(\Gamma \setminus B_n)$ is induced from the inclusion map $\Gamma \setminus B_{n+1} \hookrightarrow \Gamma \setminus B_n$. 

Recall that the inverse limit topology is the subspace topology in the product space $ \prod_{\N} \pi_{0}(\Gamma \setminus B_n)$. Note that there are finitely many complimentary components to each $B_n$, so each coordinate in the product space is discrete. Let $U$ be a complementary component of $B_n$, then $U$ defines a basis element for the topology on the end space by \[ U^*=\{e\in E(\G) \mid \rho_n(e)=U\}.\]  If one visualizes an end as a ``way to walk to infinity" in the graph $\G$, then the projection of $e$ to the $n$th coordinate is the complimentary component of $B_n$ which contains the tail of any path to $e$. 

One can check that the topology described above is Hausdorff, totally disconnected, and compact by Tychonoff's theorem. The basis given above is countable, because there are only finitely many complimentary components to each $B_n$, so $E(\Gamma)$ is second countable. Geometric topologists may be familiar with the refrain, \emph{``As such, the endspace is homeomorphic to a closed subset of the Cantor set."} In fact, we will later identify the endspace of a graph with a closed subset of \emph{Baire space} in a Borel way, as a crucial step to proving our main theorem.

Let $e\in E(\G)$, and define the $i$th neighborhood of $e$, denoted $N_i(e)$, to be the component of $\Gamma \setminus B_i$ which contains $e$. We say that the end $e$ is \textbf{accumulated by loops} if $\rk(N_i(e))>0$ for all $i$, or equivalently, if $\rk(N_i(e))=\infty$ for all $i$.
The set of ends of $\G$ accumulated by loops is a closed subspace of $E(\Gamma)$, denoted by $E_{\ell}(\G)$. Together $(E(\G), E_{\ell}(\G))$ are called the \textbf{endspace pair} of $\G$. Endspace pairs $(E(\Gamma),E_{\ell}(\G))$ and $(E(\Delta),E_{\ell}(\Delta))$ are isomorphic if there is a homeomorphism $\phi:E(\Gamma) \arr E(\Delta)$ with $\phi(E_{\ell}(\Gamma))=E_{\ell}(\Delta)$. We are now ready to state the proper homotopy equivalence classification theorem for graphs.

\begin{THM}[{\cite[Theorem 2.7]{ayala1990proper}}] \label{thm:PHEcriteria}
    Two locally finite graphs $\Gamma$ and $\Delta$ are proper homotopy equivalent exactly when $\rk(\Gamma)=\rk(\Delta)$ and $(E(\Gamma),E_{\ell}(\G)) \simeq (E(\Delta),E_{\ell}(\Delta))$. 
\end{THM}

\Cref{fig:endsExTikz} shows three graphs which are not proper homotopy equivalent to each other. The first graph is rank $3$ and has end space homeomorphic to $(\{\ast\},\emptyset)$, a singleton and the empty set. The second has infinite rank and end space homeomorphic to $(\{\frac{1}{n}\}\cup \{0\},\{0\})$, with the subspace topology from $\R$. The third graph has a Cantor set of ends, with a clopen subset (so itself homeomorphic to a Cantor set) of ends accumulated by loops.

\begin{figure}[h]
    \centering
\begin{tikzpicture}

  \tikzstyle{vertex}=[circle, draw, fill=black, inner sep=0pt, minimum size=4pt]

  \begin{scope}[shift={(-5,0)}]
    \foreach \i in {0,...,5} {
      \node[vertex] (v\i) at (0,-\i) {};
    }
    \draw (v0) -- (v5);
    \draw (v0) to[out=50,in=130,looseness=20] (v0);
    \draw (v0) to[out=140,in=220,looseness=20] (v0);
    \draw (v0) to[out=40,in=-40,looseness=20] (v0);
    \draw[->] (v5) -- ++(0,-0.5);
  \end{scope}

  \begin{scope}[shift={(0,0)}]
    \foreach \i in {0,...,5} {
      \node[vertex] (w\i) at (0,-\i) {};
      \node[vertex] (a\i) at (-1,-\i) {};
      \node[vertex] (b\i) at (-2,-\i) {};
      \draw (b\i) -- (w\i);
      \draw[->] (b\i) -- ++(-0.5,0);
    }
    \draw (w0) -- (w5);
    \foreach \i in {0,...,5} {
      \draw (w\i) to[out=45,in=-45,looseness=20] (w\i);
    }
    \draw[->] (w5) -- ++(0,-0.5);
  \end{scope}

 \begin{scope}[shift={(5,0)}]
    \node[vertex] (r0) at (0,0) {};
    \node[vertex] (r1) at (-1,-1) {};
    \node[vertex] (r2) at (1,-1) {};
    \draw (r0) -- (r1) -- ++(-0.5,-1) node[vertex] (r3) {};
    \draw (r1) -- ++(0.5,-1) node[vertex] (r4) {};
    \draw (r0) -- (r2) -- ++(-0.5,-1) node[vertex] (r5) {};
    \draw (r2) -- ++(0.5,-1) node[vertex] (r6) {};

    \draw (r3) -- ++(-0.25,-1) node[vertex] (r30) {};
    \draw (r3) -- ++(0.25,-1) node[vertex] (r31) {};
    \draw (r4) -- ++(-0.25,-1) node[vertex] (r40) {};
    \draw (r4) -- ++(0.25,-1) node[vertex] (r41) {};
    \draw (r5) -- ++(-0.25,-1) node[vertex] {};
    \draw (r5) -- ++(0.25,-1) node[vertex] {};
    \draw (r6) -- ++(-0.25,-1) node[vertex] {};
    \draw (r6) -- ++(0.25,-1) node[vertex] {};

    \draw (r0) to[out=45,in=135,looseness=20] (r0);
    \draw (r1) to[out=180,in=90,looseness=20] (r1);
    \draw (r3) to[out=180,in=90,looseness=18] (r3);
    \draw (r4) to[out=0,in=90,looseness=18] (r4);
    \draw (r41) to[out=0,in=90,looseness=16] (r41);
    \draw (r31) to[out=0,in=90,looseness=12] (r31);
    \draw (r30) to[out=180,in=90,looseness=16] (r30);
    \draw (r40) to[out=180,in=90,looseness=12] (r40);

    \node at (-1.75,-3.5) {\vdots};
    \node at (-1,-3.5) {\vdots};
    \node at (-0.25,-3.5) {\vdots};
    \node at (0.25,-3.5) {\vdots};
    \node at (1,-3.5) {\vdots};
    \node at (1.75,-3.5) {\vdots};
  \end{scope}

\end{tikzpicture}
    \caption{Examples of graphs with different PHE types.}
    \label{fig:endsExTikz}
\end{figure}

\subsection{Spaces of graphs}\label{ssec:space_of_graphs}
We realize locally finite graphs as standard Borel spaces in two different ways. Then we show that $\PHEquiv$ on these spaces are Borel bireducible.

\subsubsection{Based Graphs}
We first define the set $\mathcal{G}_k$ to be the set of infinite, $k$-regular graphs with a basepoint, considered up to graph isomorphism fixing the basepoint. We topologize this set by comparing $r$-balls of the basepoint, denoted $B_r(\G)$.
We use $\isom$ to denote isomorphism fixing the base point. 
The following defines a metric on $\mathcal{G}_k$:
\[d(\Gamma_1, \Gamma_2) = \inf \left\{ \frac{1}{2^r} \mid \; B_r(\Gamma_1) \isom B_r(\Gamma_2) \right\}. \] Two graphs will be distance zero exactly when there is an isomorphism between them fixing the basepoint, making the topology Hausdorff. 
The $\epsilon$-balls of the basepoints of any two graphs in $\mathcal{G}_k$ will be isomorphic for all $\e<\frac12$. Thus, the distance function is always finite valued and $\mathcal{G}_k$ is finite diameter.
The sets \[\mathcal{U}_{(\Gamma,r)}=\{\Delta \in \cG_k \mid B_r(\G) \isom B_r(\Delta)\}, \]  form a clopen basis for the topology. The basis is clopen because the sets $\cU_{(\G,r)}$ and $\cU_{(\G,r+\e)}$ will agree when $[r,r+\e]$ does not contain an integer or half-integer. Note that it is again very important that our graphs be connected, otherwise this topology will not be Hausdorff. 

The space $\mathcal{G}_1$ is empty, and $\mathcal{G}_2$ is a singleton, corresponding to the graph which is homeomorphic to a line. The space $\cG_3$ is our first infinite, non-discrete, Polish space, as are $\cG_k$ for all $k\geq 4$. However, none of these spaces contain all PHE-types of infinite graphs. For example, there is only one $k$-regular tree for each $k$, and for $k\geq 3$ these all have end space homeomorphic to the Cantor set. So, we also consider the spaces $\cG_{\leq k}$, which consists of infinite graphs where each vertex has degree less than or equal to $k$. We topologize the space in the exact same way as $\cG_k$ above, allowing $r=0$, so that graphs with different degree basepoints are distance $1$ from each other. 
Each space $\cG_{\leq k}$ for $k\geq 3$ contains every proper homotopy equivalence type of graph. To see this, start with the construction of standard models in Definition 2.5 of \cite{AB2021} and possibly apply proper homotopy equivalences as in \cref{fig:PHE_Valence} to decrease the degree of any vertex. The final space we consider is $\cG_{<\infty}$, which is the space of all locally finite, infinite graphs, again topologized in the same way. We will use the notation $\cG_{\star}$ to denote any of the previously described spaces of graphs. 

The degree of a vertex is not invariant under proper homotopy equivalence, so while we study the equivalence relation of PHE, note that the homotopies do not live in $\cG_{\star}\times [0,1]$.

\begin{figure}[h]
    \centering
    \begin{tikzpicture}
\tikzstyle{vertex}=[circle, draw, Turquoise, fill=Turquoise, inner sep=2pt, minimum size=4pt]
    \begin{scope}
        \node[vertex](center1) at (0,0) {};
        \draw[thick] (center1) to[out=45, in=135, looseness=30] (center1);
        \draw[thick] (center1) to[out=210-45, in=210+45, looseness=30] (center1);
        \draw[thick] (center1) to[out=295, in=25, looseness=30] (center1);
    \end{scope}
    
    \node at (2,0) {$\sim$};
    
    \begin{scope}[xshift=4cm]
        \node[vertex] (center2) at (0,0) {};
        \node[vertex] (top) at (0,.75) {};
        \node[vertex] (left) at (-.75,-.5) {};
        \node[vertex] (right) at (.75,-.5) {};
        
        \draw[thick, Turquoise] (center2) -- (top);
        \draw[thick, Turquoise] (center2) -- (left);
        \draw[thick, Turquoise] (center2) -- (right);
        
        \draw[thick] (top) to[out=45, in=135, looseness=30] (top);
        \draw[thick] (left) to[out=210-45, in=210+45, looseness=30] (left);
        \draw[thick] (right) to[out=295, in=25, looseness=30] (right);
    \end{scope}
\end{tikzpicture}
    \caption{Two graphs which are proper homotopy equivalent. The left graph is $6$-regular, while right graph is $3$-regular.}
    \label{fig:PHE_Valence}
\end{figure}

\subsubsection{Marked Graphs}
We now define a space of marked graphs. We first say that $\G$ is \emph{a graph on $\N$} if the vertex set is labeled by $\N$. A \textit{directed} graph on $\N$ is uniquely represented by an element $\Gamma \in \N^{\N^2}$, where $\Gamma_{(n,m)} \in \N$ is the number of directed edges from $n$ to $m$. 
If $\G_{(n,m)}=\G_{(m,n)}$ for all $(n,m)\in \N^2$, then we consider $\G$ as an undirected graph. 
Let $\G[\{n_1,\dots,n_k\}]$ denote the induced subgraph of $\G$ on the vertex set $\{n_1,\dots,n_k\}$, and $\Gamma[N]:=\G[\{0,1,\dots,N\}]$. 

The product topology on $\N^{\N^2}$ is Polish, and by re-indexing is homeomorphic to Baire space.
Because basic open sets in $\N^{\N^2}$ are sets where finitely many of the coordinates are specified, a neighborhood basis about the point $\G$ is given by sets of graphs with equal induced subgraphs on finite sets of vertices. Following the notation above, we name basic open sets about $\G$ by  \begin{align*}
\cV_{\{n_1,\dots, n_k\}}(\G)&=\{\Delta \mid \Delta[\{n_1,\dots,n_k\}]=\G[\{n_1,\dots,n_k\}] \}, \text{ and }\\
\cV_{[N]}(\G)&=\{\Delta \mid \Delta[N]=\G[N] \}.
\end{align*}

We define the set $\mathcal{MG}_k \subseteq \N^{\N^2}$ to be the set of undirected, $k$-regular graphs on $\N$ with $\G[N]$ connected for each $N\in \N$. We call the condition that $\G[N]$ is connected for all $N\in \N$, \textbf{strongly connected}.  If one thinks about building a graph by attaching vertices in numerical order, then this condition guarantees that at every step the graph is connected. 

Similarly, we define $\mathcal{MG}_{\leq k}$ to be the undirected, strongly connected, graphs with vertex degree bounded by $k$, and $\MG_{<\infty}$ to be the undirected, strongly connected, locally finite graphs. We again use the notation that $\MG_{\star}$ denotes any of the aforementioned spaces.

\begin{PROP}
    Each space $\mathcal{MG}_{k}$ and $\mathcal{MG}_{\leq k}$ is a $G_\delta$ subset of $\N^{\N^2}$, and hence is itself a Polish space, with the induced subspace topology.
\end{PROP}

\begin{proof} We check that each restriction on the type of graph is a $G_{\delta}$ condition within $\N^{\N^2}$, so that each of the named spaces inherits a Polish topology.
    \begin{enumerate}
        \item For fixed $n$ and $m$, the set $\{\Gamma \in  \N^{\N^2} : \Gamma_{(n,m)} = \Gamma_{(m,n)}\}$ is open because it is determined by the induced subgraph on $\{n,m\}$. 
        Undirected graphs are the intersection of these sets over all $(n,m)\in \N^2$, so form a $G_{\delta}$ set.
        \item Similarly, for each $N \in \N$, $\{\Gamma \in  \N^{\N^2} : \Gamma[N] \text{ is connected}\}$ is open, and strong connectedness is a $G_\delta$ condition.
        \item The set of graphs where vertex $n$ has degree at least $k$ can be written as \[\{\Gamma \in  \N^{\N^2} : \sum_{m \in \N} \Gamma_{n,m} > k \},\] and is open because it is achieved on a finite induced subgraph.
        So, the complimentary set, where vertex $n$ has degree bounded by $k$, is closed.
        Hence, the set of graphs where vertex $n$ has degree $k$ is $G_\delta$, as it is the intersection of an open set and a closed set.
        Graphs with uniform degree can be written \[ \bigcap_{n \in \N} \left\{\Gamma \in  \N^{\N^2} : \sum_{m \in \N} \Gamma_{n,m} = k \right\},\] which is $G_{\delta}$. On the other hand, the set of graphs with bounded degree is an intersection over closed sets, so closed.
        \qedhere
    \end{enumerate}
\end{proof}

\begin{OBS}
    The space $\mathcal{MG}_{<\infty}$ is an $F_{\sigma\delta}$ subset of $\N^{\N^2}$, and hence is itself a standard Borel space with the induced subspace topology \cite[Theorem 11.23]{Tserunyan}.
\end{OBS}

\begin{proof}
    The set of locally finite graphs can be written as \[\bigcap_{n \in \N} \bigcup_{k \in \N} \{\Gamma \in  \N^{\N^2} : \sum_{m \in \N} \Gamma_{n,m} \le k \}.
        \qedhere
    \]
\end{proof}

There is a surjective forgetful map from $F_{\star}:\mathcal{MG}_{\star} \arr \cG_{\star}$ that labels $0$ as the basepoint and forgets all other vertex labels.  

\begin{PROP}\label{forgetful_map}
    The forgetful map $F_{\star}:\mathcal{MG}_{\star} \arr \cG_{\star}$ is continuous and open.
\end{PROP}

\begin{proof}
    First observe that the preimage of a based graph is the set of all isomorphic based graphs with a legal labeling, meaning a labeling so that the graph is strongly connected. 

    Let $U=\mathcal{U}_{\Delta,r}$ be a basic open set in $\cG_{\star}$ and let $\Gamma$  be a point in $F_{\star}\inv(U)$. Then $\Gamma$ has a legal labeling of $F_{\star}(\Gamma)\in U$. Let $I\subset \N$ be the finite subset of integers that label the points in $B_{r+1}(F_{\star}(\Gamma))$. Then the open set $\cV_I(\G)$ maps into $U$ by the forgetful map.

    To see that $F_\star$ is open, fix a basic open set $\cV_I(\Gamma)$.
    Let $\Delta \in F_\star(\cV_I(\G))$. Then for some $r>0$, the ball $B_r(\Delta)$ contains all the vertices coming from $I$. 
    Hence, $\mathcal{U}_{\Delta,r+1} \subseteq F_\star(\cV_I(\Gamma))$.
    \qedhere
\end{proof}

We also show that $F_\star$ has a continuous right-inverse.

\begin{PROP}\label{rememberful_map}
    There is a continuous map $R_{\star}:\cG_{\star} \to \MG_{\star}$ so that $F_{\star}\circ R_{\star}=\id_{\cG_{\star}}$.
\end{PROP}

\begin{proof}

    First we realize the spaces $\cG_{k}$ as inverse limits. 
    For a based graph $(G,p)$ we can define its \textbf{radius} as \[\operatorname{rad}(G,p)=\sup_{x\in G} d(p,x)\] where $d$ is the path metric. 
    When $G$ is finite, the supremum is realized and is an integer or half-integer. 
    Call a vertex $v$ of $G$ an \emph{interior vertex} if $d(v,p)<\rad(G,p)$, and a \emph{border vertex} otherwise.
    We define the space $G_{k,i}$ to be the set of based graphs of radius $i$ with interior vertices of degree $k$, border vertices with degree $\leq k$, and at least one border vertex of degree $<k$.
    Endow $G_{k,i}$ with the discrete topology.
    The maps $f_{ij}:G_{k,j}\arr G_{k,i}$ for $i\leq j$ defined by $(G,p) \mapsto (B_i(G),p)$ form an inverse limit system with the spaces $G_{k,i}$. This inverse limit is exactly $\cG_k$. 

    Now we can define the map $R_{k}$ inductively on the spaces $G_{k,i}$.  Starting with $i=0$, the space $G_{k,0}$ is a singleton consisting of the graph which is a single vertex. This vertex becomes the basepoint of any graph $\G\in \cG_{k}$, so label it by $n=0$. 
    Next the space $G_{k,1}$ has finitely many graphs in it; for each one assign a labeling of the vertices with the basepoint labeled $0$, so that the labeling extends $G_{k,0}$ and uses successive integers.
    For the inductive step label each graph $G\in G_{k,i}$ so that it extends the labels assigned to the graph $f_{i(i-1)}(G)=B_{i-1}(G)$ with successive integers. 
    Because each space $G_{k,i}$ is discrete, the resulting labeling is continuous.

    Observe that by changing the degree restriction on interior vertices, each space $\cG_{\star}$ is an analogous inverse limit of discrete spaces. 
    So, we can still assign labels respecting the inverse limit as above and the resulting map will be continuous.   
\end{proof}

\subsection{Complexity of Borel equivalence relations}\label{ssec:BorelEquiv}

Our main result, \Cref{thm:mainIntro}, discusses the complexity of the proper homotopy equivalence relation on the spaces $\cG_{\star}$.
To formalize this, we need to introduce the notion of Borel reducibility of equivalence relations.
For an in depth introduction to this area, we refer the reader to \cite[Part 4]{Tserunyan}.

First, we view an equivalence relation $\sim$ on a standard Borel space $X$ as a subset $E_{\sim}$ of the product space $X^2$ as follows: \[E_{\sim}=\{(x,y)\in X^2 \mid x\sim y\}.\] 
This allows us to describe equivalence relations with set descriptors, such as Borel or analytic. Recall that a set is \emph{analytic} if it is the continuous image of a Borel set.
A consequence of \Cref{thm:mainIntro} is that the proper homotopy equivalence relation is analytic, but not Borel. This already gives a coarse notion of complexity for equivalence relations, and the notion of Borel reducibility refines it.

For equivalence relations $E$ and $F$ on standard Borel spaces $X$ and $Y$, we say that $E$ is \textbf{Borel reducible} to $F$, and we write $E \leq_B F$, if there is a Borel function $\phi: X \to Y$ so that $(x,y) \in E$ if and only if $(\phi(x),\phi(y)) \in F$.
In other words, to check whether two points are $E$-equivalent, we can pass the points to $Y$ via a Borel function and check whether the images are $F$-equivalent.
In this case, we say that $\phi$ is a \textit{Borel reduction} from $E$ to $F$.
We say that $E$ is \textbf{Borel bireducible} with $F$ if $E \leq_B F$ and $F \leq_B E$. The reducing function is only required to be Borel, because any two uncountable Polish spaces are Borel isomorphic, so the existence of a reduction does not depend on the specific topologies chosen. 

For example, \Cref{forgetful_map} and \Cref{rememberful_map} show the following.

\begin{PROP}\label{graph_spaces_are_bireducible}
        $\PHEquiv$ on $\cG_{\star}$ is Borel bireducible with $\PHEquiv$ on $\MG_{\star}$.
\end{PROP}

\begin{proof}
    By construction, the Borel maps $F_{\star}$ and $R_{\star}$ simply label or unlabel vertices. 
    Hence, $F_\star$ and $R_\star$ preserve graph isomorphism, and in particular proper homotopy equivalence.
\end{proof}


An equivalence relation $E$ is called \textit{concretely classifiable} (or \textit{smooth}) if there is a Borel reduction from $E$ to the equality relation on some Polish space $Y$. For example, similarity of $n \times n$ matrices is smooth since similarity is determined by computing the Jordan canonical form, hence giving a reduction to equality on $\R^{n^2}$. However, many notable equivalence relations are not smooth, including all of the equivalence relations mentioned in \Cref{notable_equivalence_rels}. 
We say that $E$ is \textit{classifiable by countable structures} if for some countable language  $\mathcal{L}$ (in first-order logic), there is a Borel reduction from $E$ to the isomorphism relation of the space of $\mathcal{L}$-structures with universe $\N$.

By \cite{BK96}, there is an orbit equivalence relation, denoted $C_\infty$, induced by a Borel action of $S_\infty$ (the Polish group of all permutations of $\N$) that is ``complete" in the sense that any orbit equivalence relation induced by a Borel action of a closed subgroup of $S_\infty$ is Borel reducible to $C_\infty$.
In the literature, analytic equivalence relations which are bireducible with $C_\infty$ are often referred to as \textbf{Borel complete}. 
An important class of analytic equivalence relations that are reducible to $C_\infty$ are those that are classifiable by countable structures: if $\mathcal{L}$ is a countable language, then the isomorphism relation of the space of $\mathcal{L}$-structures can be viewed as an $S_\infty $ orbit equivalence relation (see, e.g. \cite[16.C]{Kechris:classical_DST}).  

We will see in \Cref{thm:main} that the proper homotopy equivalence relation on each of the spaces $\mathcal{G}_{\star}$ and $\mathcal{MG}_{\star}$ is Borel complete.
Our result heavily relies on the following result of Camerlo and Gao.

\begin{THM}[\cite{CG01}]\label{thm:CarmeloGao}
    The homeomorphism relation on closed subsets of the Cantor space is Borel bireducible with $C_{\infty}$.
\end{THM}

\section{Results for graphs}\label{sec:results_graphs}
For clarity of exposition, we first show our main results for the spaces of $3$-regular graphs. Then in \Cref{ssec:related_results_graphs} we give modifications of these proofs that work in the $k$-regular, bounded degree, and locally finite settings.

\subsection{Genericity the Loch Ness monster PHE class}

We show that in both $\cG_3$ and $\MG_3$, the generic graph is proper homotopy equivalent to the Loch Ness monster graph.

\begin{PROP}\label{prop:3generic}
    The equivalence class of one-ended graphs with infinite rank is comeager in $\cG_3$.
\end{PROP}

\begin{proof}
First, we show that the set of graphs in $\cG_3$ with infinite rank is a dense $G_{\delta}$ set. 
Let $U_n$ be the graphs in $\cG_3$ with rank at least $n$. Each set $U_n$ is open because for any $\Gamma \in U_n$, there is an $r \in \N$ such that $B_r(\Gamma)$ already has rank $n$. Then $\mathcal{U}_{(\Gamma,r)} \subseteq U_n$.
To see that $U_n$ is dense, we show that for any $\Gamma \in \cG_3$ and $r \in \N$, we can find a graph $\Delta_U$ with rank at least $n$ and $B_r(\Delta_U) \isom B_r(\G)$. 
Such a graph $\Delta_U$ can be constructed from $\G$ by subdividing $n$ edges of $\G$ outside of $B_r(\G)$ and attaching lollipop graphs to the new vertices. We finish the claim by observing that $\bigcap U_n$ is exactly the set of graphs of infinite rank.

Now we verify that the generic $\Gamma$ is one-ended. 
Let $V_n$ be the set of $\Gamma \in \cG_3$ such that each pair of vertices in the sphere of radius $n$ of $\Gamma$ is path-connected outside of $B_n(\Gamma)$. 
If a graph $\Gamma$ is one ended, then eventually $\G \setminus B_r(\G)$ has one connected component, so the set of one-ended graphs is $\bigcup_{i\in\N}\bigcap_{n\geq i}V_n$. 
Each $V_n$ is open because for $\G \in V_n$ there is a finite path between any two vertices on the $n$-sphere. All these paths live inside some $B_r(\G)$, for $r\geq n$, and so $\cU_{(\G,r)}\subset V_n$. 
To see that $V_n$ is dense, we construct a graph $\Delta_V$ from any $\Gamma\in \cG_3$ so that $\Delta_V \in V_n$ with $B_r(\Delta_V) \isom B_r(\Gamma)$. For each pair of connected components of $\G \setminus B_r(\G)$ subdivide an edge in each component and add an edge between the new vertices, see \cref{fig:1EndDense}. 
Because there are finitely many vertices on the $n$-sphere of $\G$, there are finitely many connected components, and the modifications can be done successively in any order to construct the desired graph $\Delta_V$. 

Hence, for each $i \in \N$, $\bigcap_{n \ge i} V_n$ is dense $G_\delta$, so the set of one-ended graphs is a countable union of dense $G_\delta$ sets.
Thus, the generic $\Gamma \in \cG_3$ is one-ended. 
 
  \begin{figure}[h]
    \centering
    \includegraphics[width=8cm]{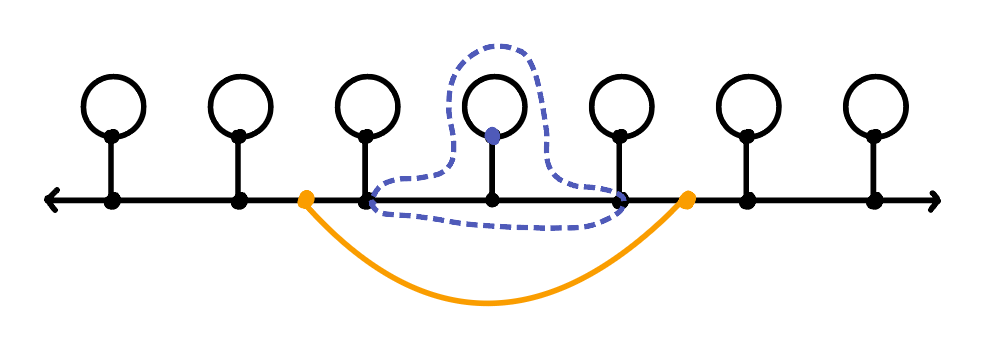}
    \caption{For the black graph $\G$, the ball of radius $2$ is drawn in blue. To find a one-ended graph in $\cU_{(\G,2)}$ we add the orange edge.}
    \label{fig:1EndDense}
\end{figure}

Because having infinite rank and one end are each generic properties and determine a unique PHE class, we see that the equivalence class of the Loch Ness monster graph generic.   
\end{proof}

\begin{COR}\label{cor:3generic_for_markedgraphs}
    The equivalence class of one-ended graphs with infinite rank is comeager in $\mathcal{M}\cG_3$.
\end{COR}

\begin{proof}
    By \Cref{prop:3generic}, the Loch Ness monster equivalence class in $\cG_3$ is comeager and contains a dense $G_\delta$ set $Z$.
    Since the forgetful map $F_3 : \MG_3 \to \cG_3$ is continuous and open (\Cref{forgetful_map}), $F_3^{-1}(Z)$ is dense $G_\delta$ in $\MG_3$.
    Since $F_3$ preserves PHE, $F_3^{-1}(Z)$ is contained in the equivalence of the Loch Ness monster graph in $\MG_3$.
\end{proof}

\subsection{Borel complexity of PHE}
We begin this subsection by reducing the homeomorphism relation on closed subsets of the Cantor set to the PHE relation on $\cG_3$.
We achieve this by (canonically) realizing every closed subset of the Cantor set as the endspace of a graph in $\cG_3$.

Below, $\mathcal{K}(2^{\mathbb{N}})$ denotes the hyperspace of compact subsets of $2^{\mathbb{N}}$ with the Vietoris topology, for more details see, e.g., \cite[Chapter 3.D]{Tserunyan}.

\begin{PROP}\label{lem:realizing_closed_subsets}
The homeomorphism relation on closed subsets of the Cantor set Borel reduces to the proper homotopy equivalence relation on $\cG_3$. That is, \[(\mathcal{K}(2^\N) , \Homeo) \le_B (\cG_3 ,\PHEquiv)\]
\end{PROP}

\begin{proof}
    We show that for any closed subset $C$ of the Cantor space, $C$ is canonically realized as the end space accumulated by loops of some $\Gamma_C \in \cG_3$. Moreover, $\Gamma_C$ can be chosen to have all of its ends accumulated by loops.
    Hence, the PHE-type of $\Gamma_C$ is exactly determined by $C$ up to homeomorphism, so $C$ and $C'$ are homeomorphic subsets of $2^\N$ if and only if $\Gamma_C$ and $\Gamma_{C'}$ are proper homotopy equivalent.

     More formally, we define a map $\mathcal{K}(2^\N) \mapsto \cG_3$, $C \mapsto \Gamma_C$ as follows.
     Let $C \subseteq{2^\N}$ be closed. Let $\G_C^*$ be the subgraph of the rooted binary tree with endspace $C$ and no leaves; that is, the graphical realization of the pruned tree corresponding to $C$. 
     Note that every vertex in $\Gamma_C^*$ has valence $2$ or $3$, aside from the root vertex which has valence $1$ or $2$. 
     Let $\Gamma_C$ be the graph obtained by subdividing each edge of $\G_C^*$,
     adding a loop to the root vertex if it was of degree $1$,
     and then adding ``lollipop graphs'' to each vertex of degree two, as illustrated in \cref{fig:lollipop}. Call the original root vertex the basepoint.
     These modifications do not add or collapse any ends, so the resulting graph $\Gamma_C \in \cG_3$ has $E({\Gamma_C}) = E_{\ell}({\Gamma_C}) \simeq C$, and is of infinite rank.

\begin{figure}[h]
    \centering
\begin{tikzpicture}
  \tikzstyle{vertex}=[circle, fill=black, inner sep=2pt, minimum size=4pt]
  \tikzstyle{bvertex}=[circle, fill=Blue, inner sep=2pt, minimum size=4pt]
  \tikzstyle{gvertex}=[circle, fill=Peach, inner sep=2pt, minimum size=4pt]
  \tikzstyle{lvertex}=[circle, fill=Turquoise, inner sep=2pt, minimum size=4pt]
 
 \begin{scope}[xshift=-5cm]
    \node[vertex] (r) at (0,0) {};
    \node[bvertex] (r0) at (-1,-1) {};
    \node[vertex] (r1) at (1,-1) {};
    \node[bvertex] (r00) at (-1.5,-2) {};
    \node[vertex] (r01) at (-.5,-2) {};
    \node[vertex] (r000) at (-1.75,-3) {};
    \node[vertex] (r001) at (-1.25,-3) {};
    \node[vertex] (r011) at (-.25,-3) {};
    \node[vertex] (r11) at (1.5,-2) {};
    \node[vertex] (r111) at (1.75,-3) {};

    \draw (r) -- (r1) -- (r11) -- (r111);
    \draw (r) -- (r0) -- (r00) -- (r000);
    \draw (r0) -- (r01) -- (r011);
    \draw (r00) -- (r001);

    \node at (-1.75,-3.5) {\vdots}; 
    \node at (-1.25,-3.5) {\vdots}; 
    \node at (-0.25,-3.5) {\vdots}; 
    \node at (1.75,-3.5) {\vdots}; 
  \end{scope}

  \begin{scope}
    \node[vertex] (r) at (0,0) {};
    \node[bvertex] (r0) at (-1,-1) {};
    \node[vertex] (r1) at (1,-1) {};
    \node[bvertex] (r00) at (-1.5,-2) {};
    \node[vertex] (r01) at (-.5,-2) {};
    \node[vertex] (r000) at (-1.75,-3) {};
    \node[vertex] (r001) at (-1.25,-3) {};
    \node[vertex] (r011) at (-.25,-3) {};
    \node[vertex] (r11) at (1.5,-2) {};
    \node[vertex] (r111) at (1.75,-3) {};
    
    \draw (r) -- node[gvertex] (g1) {} (r1) -- node[gvertex] (g2) {} (r11) -- node[gvertex] (g3) {} (r111);
    \draw (r) -- node[gvertex] (g4) {} (r0) -- node[gvertex] (g5) {} (r00) -- node[gvertex] (g6) {} (r000);
    \draw (r0) -- node[gvertex] (g7) {} (r01) -- node[gvertex] (g8) {} (r011);
    \draw (r00) -- node[gvertex] (g9) {} (r001);

    \node at (-1.75,-3.5) {\vdots}; 
    \node at (-1.25,-3.5) {\vdots}; 
    \node at (-0.25,-3.5) {\vdots}; 
    \node at (1.75,-3.5) {\vdots}; 
  \end{scope}

    \begin{scope}[xshift=5cm]
    \node[vertex] (r) at (0,0) {};
    \node[bvertex] (r0) at (-1,-1) {};
    \node[vertex] (r1) at (1,-1) {};
    \node[bvertex] (r00) at (-1.5,-2) {};
    \node[vertex] (r01) at (-.5,-2) {};
    \node[vertex] (r000) at (-1.75,-3) {};
    \node[vertex] (r001) at (-1.25,-3) {};
    \node[vertex] (r011) at (-.25,-3) {};
    \node[vertex] (r11) at (1.5,-2) {};
    \node[vertex] (r111) at (1.75,-3) {};

    \draw (r) -- node[gvertex] (g1) {} (r1) -- node[gvertex] (g2) {} (r11) -- node[gvertex] (g3) {} (r111);
    \draw (r) -- node[gvertex] (g4) {} (r0) -- node[gvertex] (g5) {} (r00) -- node[gvertex] (g6) {} (r000);
    \draw (r0) -- node[gvertex] (g7) {} (r01) -- node[gvertex] (g8) {} (r011);
    \draw (r00) -- node[gvertex] (g9) {} (r001);

    \foreach \i in {1,2,3,7,8,9} {
        \node[lvertex] (l\i) at ($(g\i) + (0.25, 0.25)$) {};
        \draw[Turquoise] (g\i) -- (l\i);
        \draw[Turquoise] (l\i) to[out=0,in=90,looseness=7] (l\i);
    }
    
    \foreach \i in {4,5,6} {
        \node[lvertex] (l\i) at ($(g\i) + (-0.25, 0.25)$) {};
        \draw[Turquoise] (g\i) -- (l\i);
        \draw[Turquoise] (l\i) to[out=180,in=90,looseness=7] (l\i);
    }

     \foreach \i in {1,11,01} {
        \node[lvertex] (b\i) at ($(r\i) + (0.25, 0.25)$) {};
        \draw[Turquoise] (r\i) -- (b\i);
        \draw[Turquoise] (b\i) to[out=0,in=90,looseness=7] (b\i);
    }

    
    \node[lvertex] (t) at ($(r) + (0, 0.37)$) {};
        \draw[Turquoise] (r) -- (t);
        \draw[Turquoise] (t) to[out=45,in=135,looseness=7] (t);

    \node at (-1.75,-3.5) {\vdots}; 
    \node at (-1.25,-3.5) {\vdots}; 
    \node at (-0.25,-3.5) {\vdots}; 
    \node at (1.75,-3.5) {\vdots}; 
\end{scope}
  
\end{tikzpicture}
 \caption{Modifying a graph $\Gamma_C^*$ (left) so that it is in $\cG_3$ by subdividing edges (middle), and then adding lollipops to degree $2$ vertices (right).}
    \label{fig:lollipop}
\end{figure}

We claim that this map is continuous.
Indeed, fix any basic open set $\cU_{(\Gamma,r)}\subset \cG_3$, and let $C \in \mathcal{K}(2^\N)$ be a point in the preimage, assuming it is nonempty.
Then every point $C' \in \mathcal{K}(2^\N)$ that agrees with $C$ on $2^{\le r}$ is also in the preimage. That is, the preimage contains the open set
\[
\{ C' \in \mathcal{K}(2^\N) : \text{ for each } v \in 2^{\le r}, C'\cap [v] = \emptyset \iff C \cap [v] = \emptyset \}.
\qedhere
\]
\end{proof}

Before showing the other reduction, we show that given $\Gamma \in \mathcal{MG}_\star$, we can canonically represent the endspace pair $(E(\Gamma),E_\ell(\Gamma))$ as sets of infinite branches of a tree on $\N$.

\begin{PROP}\label{prop:end_space_map_construction}
    There are canonical maps $E_\star: \MG_\star \to \mathcal{K}(\N^\N)$, $\Gamma \mapsto C_\Gamma$ and $E_{\ell,\star}: \MG_\star \to \mathcal{K}(\N^\N)$, $\Gamma \mapsto K_\Gamma$ such that $\left(C_\Gamma,K_{\G}\right)$ is pairwise homeomorphic to $\left(E(\Gamma),E_\ell(\Gamma)\right)$.
\end{PROP}

\begin{proof}
       Fix $\Gamma \in \mathcal{MG}_\star$. 
   Throughout, $\Gamma[>N]$ denotes the induced subgraph of $\Gamma$ on the vertices larger than $N$.
   We construct a tree $\tau_\Gamma$ on $\N$ as follows, starting with the root $\emptyset$.

   We will label subsets of $\Gamma$ with $U_x$ for finite strings $x \in \N^{<\N}$.
   First, let $U_\emptyset \defeq \Gamma$.
     Let $U_i$ be the $i^{\text{th}}$ connected component of $\Gamma[>0]$ (for each $i< m < \infty$), where the ordering of the $U_i$ is inherited from the usual order on the vertices of $\Gamma$. 
    Similarly, assuming we've constructed the $U_x$ for each $|x|=n$, 
    each $U_{x}$ splits into finitely many connected components when we remove vertex $n$. 
     Let $U_{x^\smallfrown i}$ be the $i^{\text{th}}$ component of $U_x$.
    In fact, at most one component will ``split" into more than one component.

    Let $\tau_\Gamma$ be the tree on $\N$ with vertices $x$ for each $U_x$ constructed above. 
    Let $\sigma_{\G}$ be the subtree of $\tau_{\G}$ with $x\in \sigma_{\Gamma}$ when $\rk(U_x)>0$. 
    The sets of infinite branches in $\tau_\Gamma$ and $\sigma_\Gamma$, denoted $C_\Gamma$ and $K_\Gamma$ respectively, are pairwise homeomorphic to the endspace pair $(E(\Gamma), E_\ell(\Gamma))$ by construction. 
\end{proof}

\begin{RMK}
    We note that for $\Gamma \in \mathcal{MG}_{\le k}$ or $\Gamma \in \mathcal{MG}_{k}$, each $U_i$ splits into at most $(k-1)$-many connected components, so the end space is canonically represented as subsets of $k^\N$.
\end{RMK}

For a detailed proof that the construction in \Cref{prop:end_space_map_construction} is Borel, see \Cref{sec:appendix}.

\begin{THM}\label{thm:main}
    $(\MG_\star, \PHEquiv) \le_B (\mathcal{K}(2^\N), \text{Homeo})$.
\end{THM}

\begin{proof}
We use a similar argument to the one in \cite{CG01}. 
    Let $\mathcal{L}$ be the language of Boolean algebras $\{\vee, \wedge, -, 0, 1\}$, along with an extra constant symbol $n$, unary relation symbols $\mathfrak{C}$ and $\mathfrak{K}$, and unary function symbol $f$.
    Let $\mathcal{T}$ be the axioms of Boolean algebras along with a sentence asserting that $\mathfrak{C}$ and $\mathfrak{K}$ are disjoint, and that $f$ is a surjection from $\mathfrak{C}$ to $\mathfrak{K}$.
    Let $\mathcal{M}$ be the space of models of $\mathcal{T}$ with universe $\N$.
    
    To each $\Gamma \in \MG_\star$ with rank $\rk(\Gamma) \in \N \cup \{ \infty \}$, end space representation $C_\Gamma \subset \N^\N$, and end space accumulated by loops $K_\Gamma \subseteq C_\G$, we associate the countable structure $\mathcal{B}_\Gamma \in \mathcal{M}$:
    \[
    (n_\Gamma, \mathfrak{C}_\Gamma, \mathfrak{K}_\Gamma, f_\Gamma, \vee, \wedge, -, 0, 1)
    \]
    where:
    \begin{itemize}
        \item $n_\Gamma = \text{rk}_\Gamma +1$ if $\text{rk}_\Gamma \in \N $ or $n_\Gamma = 0$ if $\text{rk}_\Gamma = \infty $,
        \item $\mathfrak{C}_\Gamma$ is the Boolean algebra of clopen subsets of $C_\Gamma$,
        \item $\mathfrak{K}_\Gamma$ is the Boolean algebra of clopen subsets of $K_\Gamma$,
        \item   $f_\Gamma$ is the (surjective) map $f_\Gamma: \mathfrak{C}_\Gamma \to \mathfrak{K}_\Gamma$ defined by $Y \mapsto Y \cap K_\Gamma$,
        \item and $\vee, \wedge, -, 0, 1$ are the usual Boolean operations.
    \end{itemize}
    
    By Stone duality, $\Gamma_1 \PHEquiv \Gamma_2$ if and only if $\mathcal{B}_{\Gamma_1} \isom \mathcal{B}_{\Gamma_2}$. 
    The assignment $\Gamma \mapsto \mathcal{B}_\Gamma$ is Borel, because the assignments $\Gamma \mapsto C_\Gamma$ and $\Gamma \mapsto K_\Gamma$, from $\MG_\star \to \mathcal{K}(\N^\N)$ are Borel.
    The map $\Gamma \mapsto n_\Gamma \in \N$ is also easily seen to be Borel.
    The Boolean operations are Borel over the space $\mathcal{K}(\N^\N)$.

    It remains to verify that the assignments $\Gamma \mapsto \mathfrak{C}_\Gamma$ and $\Gamma \mapsto \mathfrak{K}_\Gamma$ are Borel.
    As in \cite{CG01}, we form a canonical enumeration of the clopen subsets of $\N^\N$, denoted $\{\mathcal{O}'_i\}_{i \in \N}$.
    To disjointify $\mathfrak{C}_\Gamma$ and $\mathfrak{K}_\Gamma$, we modify the enumeration to $\{\mathcal{O}_i\}_{i \in \N}$, where $\mathcal{O}_{2i} = \mathcal{O}_{2i+1} = \mathcal{O}'_i$.
     
    We first enumerate the clopen subsets of $C_\Gamma$ and $K_\Gamma$ respectively as $\{ \mathcal{O}_{2i} \cap C_\Gamma \}_{i \in \N}$ and $\{ \mathcal{O}_{2i+1} \cap K_\Gamma \}_{i \in \N}$, 
     possibly with repetitions. 
     Now let $\rho$ be the enumeration of the sequence $\{ \mathcal{O}_{2i} \cap C_\Gamma \}_{i \in \N}$ without repetitions (but preserving the order), and let
     $\rho_\ell$ be the enumeration of the sequence $\{ \mathcal{O}_{2i+1} \cap K_\Gamma \}_{i \in \N}$ without repetitions (but preserving the order).
     Then $\rho$ and $\rho_\ell$ are Borel functions that uniquely describe $\mathfrak{C}_\Gamma$ and $\mathfrak{K}_\Gamma$.
\end{proof} 

As an immediate corollary of 
\Cref{graph_spaces_are_bireducible}, \Cref{lem:realizing_closed_subsets},  and \Cref{thm:main}, we have the following.

\begin{COR}
The equivalence relations 
$(\MG_3, \PHEquiv)$, $(\cG_3, \PHEquiv)$, and $(\mathcal{K}(2^\N), \Homeo)$ are all Borel bireducible.
\end{COR}

\subsection{Extending the results past the three-regular case}\label{ssec:related_results_graphs}
In this subsection, we extend the results for $\cG_3$ and $\MG_3$ to apply to the spaces of graphs with different degree restrictions.

\begin{PROP}[Intro \Cref{thm:GenericIntro} expanded]\label{prop:kgeneric}
     For $k\geq 3$, the equivalence class of one-ended graphs with infinite rank is comeager in each of $\cG_k$, $\cG_{\leq k}$, and $\cG_{<\infty}$.
\end{PROP}

\begin{proof}
\textbf{Infinite rank is dense $G_\delta$:}
As in \Cref{prop:3generic}, having rank at least $n$ is an open property in any of the spaces, since it is witnessed on a finite subgraph.
To see that having rank at least $n$ is dense in these spaces, it suffices to show that for any finite subgraph $G$ satisfying the degree-regularity properties, we can construct a graph $\Gamma$ of rank at least $n$ with the desired degree-regularity properties and such that $\Gamma$ has $G$ as a subgraph. 

Because $k$-regular graphs are also finite degree and degree at most $k$, it suffices to show this for $\cG_k$.
For $\cG_k$, to construct $\Delta_U^k$ as in \Cref{prop:3generic}: subdivide $n$ edges of $\G$ outside of $B_r(\G)$ and attach the following graphs to the new vertices. 
    \begin{itemize}
        \item For even $k$, attach $k-2$ self loops.
        \item For odd $k$, attach a \emph{barrel cactus graphs} as in \cref{fig:Oddk}. That is, for each vertex from the subdivision, add another vertex, connect them by $k-2$ edges, and add a self-loop to the new vertex.   
    \end{itemize}

\textbf{One-endedness is generic:}
As in \Cref{prop:3generic}, let $V_n$ be the set of $\Gamma \in \cG_\star$ such that each pair of vertices in the sphere of radius $n$ of $\Gamma$ is path-connected outside of $B_n(\Gamma)$. 
Again as in \Cref{prop:3generic}, membership to $V_n$ is witnessed on a finite subgraph, so $V_n \subseteq \cG_\star$ is open.

To see that $V_n$ is dense in $\cG_\star$, we start with any graph $\G\in \cG_k$ and construct $k$-regular graphs $\Delta_V^k \in V_n$ with $B_r(\G) = B_r(\Delta_V^k)$.
To construct $\Delta_V^k$: For each pair of connected components of $\G \setminus B_r(\G)$ subdivide an edge in each and add $k-2$ edges between the new vertices.  
 Because a $k$-regular graph is also in $\cG_{\leq k}$, and $\cG_{<\infty}$, this concludes the proof for all spaces $\cG_\star$.
\end{proof}

By the same proof as \Cref{cor:3generic_for_markedgraphs}, we have the following.

\begin{COR}
    For $k\geq 3$, the equivalence class of one-ended graphs with infinite rank is comeager in each of $\MG_k$, $\MG_{\leq k}$, and $\MG_{<\infty}$.
\end{COR}

\begin{PROP}\label{lem:k_realizes}
$(\mathcal{K}(2^\N) , \text{Homeo}) \le_B (\cG_k ,\PHEquiv)$ for every $k \ge 3$. 
\end{PROP}

\begin{proof}
    We begin as in \Cref{lem:realizing_closed_subsets} with $C\subset 2^{\N}$ and the graph $\G_C^*$. For each value of $k$ we describe how to modify $\G_C^*$ to get a graph in $\cG_k$ with $E(\G_C^k)=E_{\ell}(\G_C^k)=C$.
    
    For $k$ odd, perform the following modifications to $\G_C^*$: \begin{itemize}
    \item for each degree $3$ vertex add $\frac{k-3}{2}$ self-loops, 
    \item for each degree $2$ vertex add a new vertex, connect them by $k-2$ edges, and add a self-loop to the new vertex. 
    \end{itemize} If there is a degree one vertex, add $\frac{k-1}{2}$ self-loops. This process adds loops to each vertex, so ultimately every end is accumulated by loops. See \cref{fig:Oddk} for an example where $k=7$.

    \begin{figure}[h]
        \centering
                  \begin{tikzpicture}
  \tikzstyle{vertex}=[circle, fill=black, inner sep=2pt, minimum size=4pt]
  \tikzstyle{bvertex}=[circle, fill=Bittersweet, inner sep=2pt, minimum size=4pt]
  \tikzstyle{gvertex}=[circle, fill=ForestGreen, inner sep=2pt, minimum size=4pt]
  \tikzstyle{lvertex}=[circle, fill=Turquoise, inner sep=2pt, minimum size=4pt]
 
 \begin{scope}[xshift=-5cm]
    \node[vertex] (r) at (0,0) {};
    \node[bvertex] (r0) at (-1,-1) {};
    \node[vertex] (r1) at (1,-1) {};
    \node[bvertex] (r00) at (-1.5,-2) {};
    \node[vertex] (r01) at (-.5,-2) {};
    \node[vertex] (r000) at (-1.75,-3) {};
    \node[vertex] (r001) at (-1.25,-3) {};
    \node[vertex] (r011) at (-.25,-3) {};
    \node[vertex] (r11) at (1.5,-2) {};
    \node[vertex] (r111) at (1.75,-3) {};

    \draw (r) -- (r1) -- (r11) -- (r111);
    \draw (r) -- (r0) -- (r00) -- (r000);
    \draw (r0) -- (r01) -- (r011);
    \draw (r00) -- (r001);

    \node at (-1.75,-3.5) {\vdots}; 
    \node at (-1.25,-3.5) {\vdots}; 
    \node at (-0.25,-3.5) {\vdots}; 
    \node at (1.75,-3.5) {\vdots}; 
  \end{scope}

  \begin{scope}
    \node[vertex] (r) at (0,0) {};
    \node[bvertex] (r0) at (-1,-1) {};
    \node[vertex] (r1) at (1,-1) {};
    \node[bvertex] (r00) at (-1.5,-2) {};
    \node[vertex] (r01) at (-.5,-2) {};
    \node[vertex] (r000) at (-1.75,-3) {};
    \node[vertex] (r001) at (-1.25,-3) {};
    \node[vertex] (r011) at (-.25,-3) {};
    \node[vertex] (r11) at (1.5,-2) {};
    \node[vertex] (r111) at (1.75,-3) {};

    \draw (r) -- (r1) -- (r11) -- (r111);
    \draw (r) -- (r0) -- (r00) -- (r000);
    \draw (r0) -- (r01) -- (r011);
    \draw (r00) -- (r001);

     \foreach \i in {1,11,01} {
        \node[lvertex] (b\i) at ($(r\i) + (0.5, 0.5)$) {};
        \draw[Turquoise] (r\i) -- (b\i);
        \draw[Turquoise] (r\i) to[bend right=50] (b\i);
        \draw[Turquoise] (r\i) to[bend left=50] (b\i);
        \draw[Turquoise] (r\i) to[bend right=25] (b\i);
        \draw[Turquoise] (r\i) to[bend left=25] (b\i);
        \draw[Turquoise] (b\i) to[out=0,in=90,looseness=7] (b\i);
    }

    \node[lvertex] (b) at ($(r)+(0,.707)$) {};    
    \draw[Turquoise] (r) -- (b);
        \draw[Turquoise] (r) to[bend right=50] (b);
        \draw[Turquoise] (r) to[bend left=50] (b);
        \draw[Turquoise] (r) to[bend right=25] (b);
        \draw[Turquoise] (r) to[bend left=25] (b);
        \draw[Turquoise] (b) to[out=45,in=135,looseness=7] (b);
        
    \draw[Turquoise] (r0) to[out=200,in=140,looseness=9] (r0);
    \draw[Turquoise] (r0) to[out=140,in=80,looseness=9] (r0);
    \draw[Turquoise] (r00) to[out=200,in=140,looseness=9] (r00);
    \draw[Turquoise] (r00) to[out=140,in=80,looseness=9] (r00);

    \node at (-1.75,-3.5) {\vdots}; 
    \node at (-1.25,-3.5) {\vdots}; 
    \node at (-0.25,-3.5) {\vdots}; 
    \node at (1.75,-3.5) {\vdots}; 
  \end{scope}

  \end{tikzpicture}
        \caption{Modifying the graph $\G_C^*$ so that it is in $\cG_7$ by attaching ``barrel cactus graphs."}
        \label{fig:Oddk}
    \end{figure}

For $k$ even, $k\neq 4$, we first double every edge of $\Gamma_C^*$ so that every vertex has degree $2$, $4$ or $6$. Then for each vertex $v$, add $\frac{k-\deg(v)}{2}$ self-loops so that the resulting degree of $v$ is $k$.

    \begin{figure}
        \centering
                \begin{tikzpicture}
  \tikzstyle{vertex}=[circle, fill=black, inner sep=2pt, minimum size=4pt]
  \tikzstyle{bvertex}=[circle, fill=Bittersweet, inner sep=2pt, minimum size=4pt]
  \tikzstyle{gvertex}=[circle, fill=ForestGreen, inner sep=2pt, minimum size=4pt]
  \tikzstyle{lvertex}=[circle, fill=Turquoise, inner sep=2pt, minimum size=4pt]
 
 \begin{scope}[xshift=-5cm]
    \node[vertex] (r) at (0,0) {};
    \node[bvertex] (r0) at (-1,-1) {};
    \node[vertex] (r1) at (1,-1) {};
    \node[bvertex] (r00) at (-1.5,-2) {};
    \node[vertex] (r01) at (-.5,-2) {};
    \node[vertex] (r000) at (-1.75,-3) {};
    \node[vertex] (r001) at (-1.25,-3) {};
    \node[vertex] (r011) at (-.25,-3) {};
    \node[vertex] (r11) at (1.5,-2) {};
    \node[vertex] (r111) at (1.75,-3) {};

    \draw (r) -- (r1) -- (r11) -- (r111);
    \draw (r) -- (r0) -- (r00) -- (r000);
    \draw (r0) -- (r01) -- (r011);
    \draw (r00) -- (r001);

    \node at (-1.75,-3.5) {\vdots}; 
    \node at (-1.25,-3.5) {\vdots}; 
    \node at (-0.25,-3.5) {\vdots}; 
    \node at (1.75,-3.5) {\vdots}; 
  \end{scope}

 \begin{scope}
    \node[vertex] (r) at (0,0) {};
    \node[bvertex] (r0) at (-1,-1) {};
    \node[vertex] (r1) at (1,-1) {};
    \node[bvertex] (r00) at (-1.5,-2) {};
    \node[vertex] (r01) at (-.5,-2) {};
    \node[vertex] (r000) at (-1.75,-3) {};
    \node[vertex] (r001) at (-1.25,-3) {};
    \node[vertex] (r011) at (-.25,-3) {};
    \node[vertex] (r11) at (1.5,-2) {};
    \node[vertex] (r111) at (1.75,-3) {};

    \draw (r) to[bend right=15] (r1);
    \draw (r) to[bend left=15] (r1);
    
    \draw (r1) to[bend right=15] (r11);
    \draw (r1) to[bend left=15] (r11);
    
    \draw (r11) to[bend right=15] (r111);
    \draw (r11) to[bend left=15] (r111);
    
    \draw (r) to[bend right=15] (r0);
    \draw (r) to[bend left=15] (r0);
    
    \draw (r0) to[bend right=15] (r00);
    \draw (r0) to[bend left=15] (r00);
    
    \draw (r00) to[bend right=15] (r000);
    \draw (r00) to[bend left=15] (r000);
    
    \draw (r0) to[bend right=15] (r01);
    \draw (r0) to[bend left=15] (r01);
    
    \draw (r01) to[bend right=15] (r011);
    \draw (r01) to[bend left=15] (r011);
    
    \draw (r00) to[bend right=15] (r001);
    \draw (r00) to[bend left=15] (r001);

    \node at (-1.75,-3.5) {\vdots}; 
    \node at (-1.25,-3.5) {\vdots}; 
    \node at (-0.25,-3.5) {\vdots}; 
    \node at (1.75,-3.5) {\vdots}; 
  \end{scope}

  \begin{scope}[xshift=5cm]
    \node[vertex] (r) at (0,0) {};
    \node[bvertex] (r0) at (-1,-1) {};
    \node[vertex] (r1) at (1,-1) {};
    \node[bvertex] (r00) at (-1.5,-2) {};
    \node[vertex] (r01) at (-.5,-2) {};
    \node[vertex] (r000) at (-1.75,-3) {};
    \node[vertex] (r001) at (-1.25,-3) {};
    \node[vertex] (r011) at (-.25,-3) {};
    \node[vertex] (r11) at (1.5,-2) {};
    \node[vertex] (r111) at (1.75,-3) {};

    \draw (r) to[bend right=15] (r1);
    \draw (r) to[bend left=15] (r1);
    
    \draw (r1) to[bend right=15] (r11);
    \draw (r1) to[bend left=15] (r11);
    
    \draw (r11) to[bend right=15] (r111);
    \draw (r11) to[bend left=15] (r111);
    
    \draw (r) to[bend right=15] (r0);
    \draw (r) to[bend left=15] (r0);
    
    \draw (r0) to[bend right=15] (r00);
    \draw (r0) to[bend left=15] (r00);
    
    \draw (r00) to[bend right=15] (r000);
    \draw (r00) to[bend left=15] (r000);
    
    \draw (r0) to[bend right=15] (r01);
    \draw (r0) to[bend left=15] (r01);
    
    \draw (r01) to[bend right=15] (r011);
    \draw (r01) to[bend left=15] (r011);
    
    \draw (r00) to[bend right=15] (r001);
    \draw (r00) to[bend left=15] (r001);

    \draw[Turquoise] (r1) to[out=-20,in=40,looseness=9] (r1);
    \draw[Turquoise] (r1) to[out=40,in=100,looseness=9] (r1);
    \draw[Turquoise] (r11) to[out=-20,in=40,looseness=9] (r11);
    \draw[Turquoise] (r11) to[out=40,in=100,looseness=9] (r11);
    \draw[Turquoise] (r01) to[out=-20,in=40,looseness=9] (r01);
    \draw[Turquoise] (r01) to[out=40,in=100,looseness=9] (r01);

    \draw[Turquoise] (r) to[out=0,in=90,looseness=8] (r);
    \draw[Turquoise] (r) to[out=90,in=180,looseness=8] (r);
        
    \draw[Turquoise] (r0) to[out=180,in=90,looseness=8] (r0);
    \draw[Turquoise] (r00) to[out=180,in=90,looseness=8] (r00);

    \node at (-1.75,-3.5) {\vdots}; 
    \node at (-1.25,-3.5) {\vdots}; 
    \node at (-0.25,-3.5) {\vdots}; 
    \node at (1.75,-3.5) {\vdots}; 
  \end{scope}

  \end{tikzpicture}
        \caption{Modifying the graph $\G^*_C$ so that it is in $\cG_8$, first by doubling edges, then adding self-loops.}
        \label{fig:enter-label}
    \end{figure}

Finally, the case where $k=4$ requires a bit more subtlety. First double every edge of $\G_C^*$ and then perform a local PHE at every degree $6$ vertex to split it into two degree $4$ vertices, as in \cref{fig:k4}. Because these are disjoint, compactly supported, homotopy equivalences, they do not change the PHE type of the graph, and in particular do not change the end space. If the root was degree one, then add a loop after doubling so that it is degree $4$. Note that after doubling edges, every end of the graph was accumulated by loops.

\begin{figure}
    \centering
    \begin{tikzpicture}
\tikzstyle{vertex}=[circle, fill=Turquoise, inner sep=2pt, minimum size=4pt]
\node[vertex] (A) at (0,0) {};

\draw[thick] (-1,1) -- (A) -- (1,-1);
\draw[thick] (-1,-1) -- (A) -- (1,1);
\draw[thick] (0,-1.3) -- (A) -- (0,1.3);

\node[vertex] (B) at (4,.5) {};
\node[vertex] (C) at (4,-.5) {};

\draw[thick, Turquoise] (B) -- (C);

\draw[thick] (B) -- (5,1.5);
\draw[thick] (B) -- (4,1.8);
\draw[thick] (B) -- (3,1.5);

\draw[thick] (C) -- (5,-1.5);
\draw[thick] (C) -- (4,-1.8);
\draw[thick] (C) -- (3,-1.5);

\draw[-{Stealth[length=3mm, width=2mm]}, thick] (1.3,0) -- (2.7,0);

\node at (2, -0.5) {Expand};

\end{tikzpicture}
    \caption{A local PHE at a degree $6$ vertex to split it into two degree $4$ vertices. }
    \label{fig:k4}
\end{figure}

As in \Cref{lem:realizing_closed_subsets}, these maps are Borel reductions.
\end{proof}

\begin{COR}\label{cor:k_realizes}
    $(\mathcal{K}(2^\N) , \text{Homeo}) \le_B (\MG_k ,\PHEquiv)$ for every $k \ge 3$.
\end{COR}

Because $\MG_k\subset \MG_{\leq k}\subset \MG_{\infty}$ and $\cG_k\subset \cG_{\leq k}\subset \cG_{\infty}$, \Cref{lem:k_realizes} and \Cref{cor:k_realizes} apply to all spaces $\cG_\star$ and $\MG_\star$.
Since \Cref{thm:main} was proved in general for $\MG_\star$, we obtain the following.

\begin{COR}[Intro \Cref{thm:mainIntro} expanded]
    For $k\geq 3$, the proper homotopy equivalence relations on the spaces $\cG_k$, $\cG_{\leq k}$, and $\cG_{<\infty}$ are each bireducible with $C_\infty$. 
\end{COR}

\section{Results for surfaces}\label{sec:results_surfaces}

\subsection{Surfaces}\label{ssec:surfaces}

By \textbf{surface} we mean a second countable, Hausdorff, orientable, 2-manifold without boundary.
Two compact surfaces are homeomorphic if they have the same number of genus. 
The classification of non-compact surfaces predates, but is in direct analogy with, the classification of infinite graphs up to proper homotopy equivalence. It is due independently to Ker\'ekj\'art\'o \cite{Kerekjarto} and Richards \cite{Richards}. 

\begin{THM}
    Two surfaces $\Sigma$ and $\Sigma'$ are homeomorphic exactly when $g(\Sigma)=g(\Sigma')$ and \\$(E(\Sigma),E_{g}(\Sigma)) \simeq (E(\Sigma'),E_{g}(\Sigma'))$. 
\end{THM} 

In this statement, $g(\Sigma)\in \N \cup \{0,\infty\}$ is the number of genus, $E(\Sigma)$ is the end space, and $E_g(\Sigma)$ is the subspace of ends accumulated by genus. The endspace of a surface can be defined mutatis mutandis with \Cref{def:Endspace}, replacing $\G$ with $\Sigma$. To define the \textbf{ends accumulated by genus}, replace rank with genus in the definition of ends accumulated by loops.

\subsection{Spaces of Surfaces}

Next we want to realize sets of noncompact surfaces as Polish spaces; this includes infinite-type surfaces as well as finite-type surfaces with punctures. First, we define a \textbf{pair of pants} to be any compact, orientable, genus zero surface with at least one boundary component. If the pair of pants $P$ has $d+1$ boundary components we call them specifically a $d$-legged pair of pants, so that a $3$-holed sphere has two legs (the third boundary component being the waist). We allow $1$-legged pants, which are topologically annuli, and $0$-legged pants, which are topologically disks. 
A \textbf{pants decomposition} of a surface $S$ is a collection of subsurfaces $\cP=\{P_j\}_{j\in I}$ which are each themselves pants, cover the surface $S$, and overlap exactly on their boundary components. Denote the boundary components of $P_j$ by $\{\del_1P_j, \dots, \del_{d+1} P_j\}$. That is, $S$ can be constructed by gluing the pants together along their boundary components by orientation preserving homeomorphisms, and we can write 
\[S=\left(\bigsqcup_{j\in I} P_j \right)/ \left( \del_i P_j \sim \del_k P_{\ell} \right), \] 
with each $\del_i P_j$ appearing in the quotient exactly once. Note that our pants have boundary, so the only possible pants decomposition around a puncture is a sequence of annuli.

For each space of graphs $\cG_{\star}$ discussed in \Cref{ssec:space_of_graphs} we will define an analogous space of surfaces, $\cS_{\star}$. First, let $\hat{\cS_k}$ be the set of surfaces with a $(k-1)$-legged pants decomposition. Similarly, let $\hat{S}_{\leq k}$ denote the set of surfaces with a pants decomposition in which all pants have at most $k-1$ legs, and $\hat{S}_{<\infty}$ denote the set of surfaces with any pants decomposition. 
For a surface with any type of pants decomposition $(S,\cP)$ we associate a graph in the following way. For each pair of paints $P_j$, associate a vertex $v_j$. Whenever a boundary component of $P_j$ is glued to a boundary component of $P_k$, associate an edge $(v_j,v_k)$, note that $j$ and $k$ may be equal. In this way, each $(d-1)$-legged pants corresponds to a vertex of degree $d$. If we consider surfaces with a basepoint interior to the pants decomposition, then we can associate a based graph. This association defines a function $\Phi_{\star}:\hat{\cS}_{\star} \arr \cG_{\star}$ for each choice of $\star$. See \cref{fig:SurfaceToGraph} for an example where the surface is compact. 

\begin{figure}[h]
    \centering
    \begin{tikzpicture}

\node[anchor=south west,inner sep=0, scale=0.5] (image) at (0,0) {\includegraphics[width=20cm]{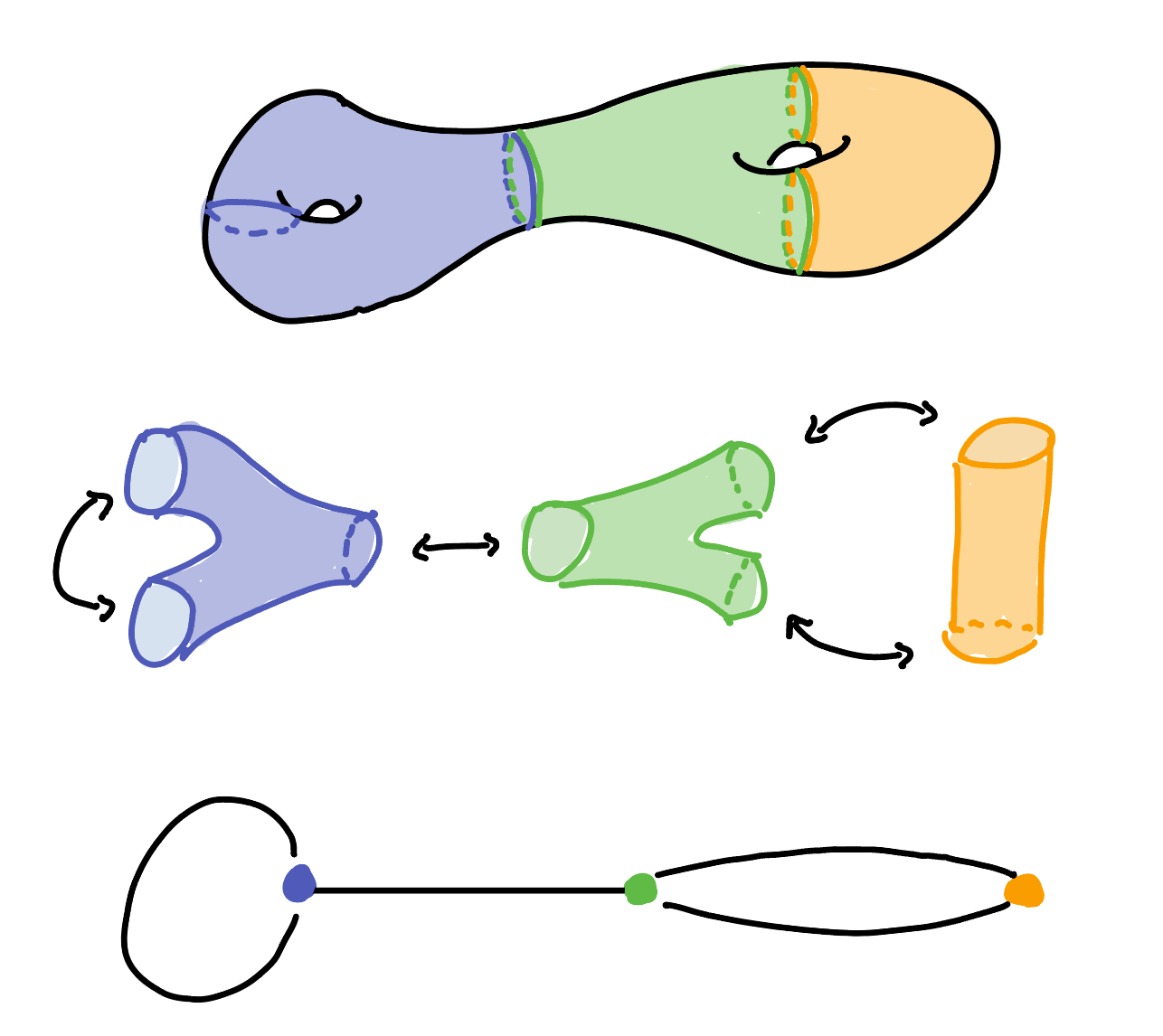}};

\node at (2.25, 8.5) {\textbf{$P_1$}};
\node at (5.5, 8.5) {\textbf{$P_2$}};
\node at (8, 8.5) {\textbf{$P_3$}};

\node at (3, 1.5) {\textbf{$v_1$}};
\node at (5.2, 1.5) {\textbf{$v_2$}};
\node at (9.25, 1.5) {\textbf{$v_3$}};

\end{tikzpicture}
    \caption{Constructing a graph from a surface with a pants decomposition.}
    \label{fig:SurfaceToGraph}
\end{figure}

Each map $\Phi_{\star}:\hat{\cS_{\star}} \arr \cG_{\star}$ is surjective, so we can then endow $\hat{\cS_{\star}}$ with the initial topology. Recall that the \emph{initial topology} is the coarsest topology which makes $\Phi_{\star}$ continuous and has the sets \[ \mathcal{B}=\{\Phi_{\star}\inv(\cU) \mid \cU \text{ is open in } \cG_{\star} \} \] as a subbase. 
Note that this topology will not be Hausdorff. For example, the same surface with two pants decompositions in which pants curves are isotopic, but not equal, are indistinguishable by open sets in the topology. 
So, we quotient by the equivalence relation induced by $\Phi_{\star}$, that is $(S,\cP,b)\sim(S',\cP',b')$ if $\Phi_{\star}(S,\cP,b)=\Phi_{\star}(S',\cP',b')$. Call these quotient spaces (with the quotient topologies) $\cS_{\star}$, and note that they are Hausdorff and homeomorphic to the corresponding space $\cG_{\star}$. For detailed explanation, see Corollary 22.3 of Munkres' Topology, \cite{munkres2000topology}. 

Points in $\cS_{\star}$ can be thought of as surfaces, up to homeomorphism, with based pants decompositions, up to homotopy. A puncture on a surface will map to a ray in the associated graph. So, similar to the discussion of which graphs live in which spaces $\cG_{\star}$, there are no surfaces with isolated punctures in $\cS_{k}$, and all homeomorphism types of noncompact surfaces are realized in $\cS_{\leq k}$ for $k\geq 3$.

Let $\Psi_{\star}$ denote the induced homeomorphism from $\cS_{\star} \arr \cG_{\star}$. By construction $\Psi_{\star}$ maps a genus to a loop, and an end to an end. As such, $\Psi_{\star}$ and it's inverse are Borel reductions between proper homotopy equivalence and homeomorphism. From the results on $\cG_{\star}$ we have the following.

\begin{COR}[Intro \Cref{cor:mainIntro} expanded]
    The equivalence relation of homeomorphism on each space $\cS_{\star}$ is bireducible with $C_{\infty}$.
    In each space $\cS_{\star}$ the equivalence class of surfaces with exactly one end which is accumulated by genus is comeager.
\end{COR}

\appendix 
\section{The end space map is Borel}\label{sec:appendix}

Here we carefully prove that the construction in \Cref{prop:end_space_map_construction} is Borel.
We first prove that the map taking a finitely branching tree on $\N$ to its set of infinite branches is Borel. 
This is well known (see, e.g., \cite[Chapter 4]{Kechris:classical_DST}), but we include the proofs here for completeness.
We remark that this map is not continuous (see, e.g., \cite[Exercise 4.32]{Kechris:classical_DST}).

\begin{PROP}
    The space $\mathcal{T}_f$ of finitely branching trees on $\N$ is a Borel subset of $2^{\N^{<\N}}$.
\end{PROP}

\begin{proof}
    $\tau \in 2^{\N^{<\N}}$ is a finitely branching tree if and only if
    \begin{enumerate}
        \item $\tau(\emptyset) = 1$.
        \item For each $x \in \N^{<\N}$ and $i \in \N$ with $\tau(x^\smallfrown i) =1$, $\tau(x)=1$.
        \item For each $x \in \N^{<\N}$, there is $N \in \N$ such that for all $i > N$, $\tau(x^\smallfrown i) = 0$.
    \end{enumerate}
    Condition (1) is open, (2) is $G_\delta$, and (3) is $\Pi^0_3$.
\end{proof}

\begin{PROP}\label{prop:avoiding_cylinder_is_G_delta}
    For any $x \in \N^{<\N}$ and corresponding cylinder set $[x]$, the set
    \[
    \mathcal{T}_{x} \defeq \{\tau \in \mathcal{T}_f : \; E(\tau) \cap [x] \neq \emptyset\}
    \]
    is Borel (in fact $G_\delta$) in $\mathcal{T}_f$.
\end{PROP}

\begin{proof}
    By K\"{o}nig's lemma, since each $\tau \in \mathcal{T}_f$ is finitely branching, $\tau$ has an infinite branch starting from $x$ if and only if the subtree below $x$ is infinite.
    Hence,
    $\mathcal{T}_x = \bigcap_{n\in\N} \bigcup_{y \supseteq x, |y|=n} \tau(y)=1$.
\end{proof}

\begin{PROP}
    The end space map $E:\mathcal{T}_f \to \mathcal{K}(\N^{\N})$ that maps $\tau$ to its set of infinite branches is Borel.
\end{PROP}

\begin{proof}
     It is enough to show that for any open set $U \subseteq \N^\N$, the sets 
    \[
    \mathcal{V}_U \defeq \{\tau \in \mathcal{T}_f : \; E(\tau) \cap U \neq \emptyset \}
    \text{ and }
    \mathcal{W}_U \defeq \{\tau \in \mathcal{T}_f : \; E(\tau) \subseteq U  \}
    \]
    are Borel.
    Write $U$ as a countable union of cylinder sets, $U = \bigcup_{n\in\N}[x_i]$, where each $x_i \in \N^{<\N}$.

    By \Cref{prop:avoiding_cylinder_is_G_delta}, $\mathcal{V}_U$ is $\Sigma_2^0$.
   Since each $E(\tau)$ is compact, \[\mathcal{W}_U = \bigcup_{n\in\N} \{\tau \in \mathcal{T}_f : \; E(\tau) \subseteq \bigcup_{i<n} [x_i] \}.\]
   Again by \Cref{prop:avoiding_cylinder_is_G_delta}, since $\bigcup_{i<n}[x_i]$ is clopen, each $\{\tau \in \mathcal{T}_f : \; E(\tau) \subseteq \bigcup_{i<n} [x_i] \}$ is $\Pi^0_2$, so $\mathcal{W}_U$ is $\Sigma^0_3$.
   Therefore, $E$ is a Borel map. 
\end{proof}

We now show that the map taking a graph in $\MG_\star$ to its corresponding finitely branching tree on $\N$ (resp., its tree with rank $k$) from \Cref{prop:end_space_map_construction} is Borel.

\begin{OBS}\label{obs:connected_conditions}
For any subset $I \subseteq \N$, and any $i,n,m\in \N$, the following subsets of $\MG_\star$ are Borel.
Below, by the ordering of the connected components, we mean the ordering induced by the usual order on $\N$.
That is, we order the components by the labels of their smallest vertices.
    \begin{enumerate}
        \item $\mathcal{R}_{I} \defeq \{\Gamma \in \MG_\star : \G[I] \text{ has positive rank}\}$
        \item $\mathcal{P}_{I,i,j} \defeq \{\Gamma \in \MG_\star : i \text{ is connected to } j \text{ in }\Gamma[\N\setminus I]\}$
        \item $\mathcal{C}_{I,n} \defeq \{\Gamma \in \MG_\star : \Gamma[\N\setminus I] \text{ has at least $n$ connected components}\}$
        \item $\mathcal{C}_{I,n,m} \defeq \{\Gamma \in \MG_\star : m \text{ is the least vertex in the $n^\text{th}$ connected component of } \Gamma[\N\setminus I]\}$
    \end{enumerate}
\end{OBS}

\begin{proof}
\begin{enumerate}
    \item $\mathcal{R}_{I}$ is open, since this condition is witnessed on finitely many vertices.
    \item $\mathcal{P}_{I,i,j}$ is open, since this condition is witnessed on finitely many vertices.
    \item $\mathcal{C}_{I,2}$ is $F_\sigma$ for any $I$, since it is the countable union of complements of the $\mathcal{P}_{I,i,j}$. 
    Similarly, each $\mathcal{C}_{I,n}$ is a countable Boolean combination of $\mathcal{P}_{I,i,j}$s.
    \item $\Gamma \in \mathcal{C}_{I,n,m}$ if and only if 
    \begin{itemize}
        \item $\Gamma \in \mathcal{C}_{I,n}$, and
        \item there are $m_1,...,m_{n-1} < m$ such that the $m_1,...,m_{n-1},m=m_n$ are pairwise disconnected outside of $I$ (that is, $\Gamma \in \bigcap_{i\neq j\le n}\mathcal{P}_{I,m_i,m_j}^c$), and
        \item for all $j < m$, $j$ is connected to one of the $m_i$ outside of $I$ (that is, $\Gamma \in \bigcup_{i< n}\mathcal{P}_{i,j,m_i}$). 
        \qedhere
    \end{itemize}
\end{enumerate}
\end{proof}

\begin{LEM}\label{end_space_maps_are_Borel}
    For each $x \in \N^{<\N}$, the set $\mathcal{H}_x \defeq \{ \G \in \MG_\star : \; x \in \tau_\G \}$ is Borel.
    Hence, the map $\Gamma \mapsto \tau_\Gamma$, $\MG_\star \to {2}^{\N^{<\N}}$ constructed in \Cref{prop:end_space_map_construction} is Borel.
    Similarly, the map $\Gamma \mapsto \tau_\Gamma^\ell$ is also Borel.
\end{LEM}

\begin{proof}
    We will induct on the length of $x$, starting with the observation that $\mathcal{H}_\emptyset = \MG_\star$. 
    We will prove the stronger statement that $\mathcal{H}_x$ is Borel, and for any $I \subseteq \N$, the set $\mathcal{H}_{x,I} \defeq \{ \Gamma \in \mathcal{H}_x : \forall i\in I, \; i\in U_x \}$ is Borel. 
    We will use \Cref{obs:connected_conditions} throughout.
    
    For $|x| = 1$, i.e., $x=n$ for some $n\in\N$, note that $\mathcal{H}_n$ is the collection of $\G \in \MG_\star$ such that $\G[>0]$ has at least $n+1$ connected components.
    This is an $F_\sigma$ set.
    Furthermore, $\mathcal{H}_{n,I}$ is the set of $\Gamma \in \mathcal{H}_n$ such that for some $m \in \N$ and for all $i \in I$, $m$ is the least vertex in the $(n+1)^\text{th}$ connected component of $\Gamma[>0]$ and each $i \in I$ is connected to $m$ in $\Gamma[>0]$.
    Again, this is Borel by \Cref{obs:connected_conditions}.

    Now suppose $\mathcal{H}_x$ is Borel with $|x|=k \ge 1$, and let $n \in \N$.
    Suppose also that for any $I \subseteq \N$, $\mathcal{H}_{x,I} \defeq \{ \Gamma \in \mathcal{H}_x : \forall i\in I, \; i\in U_x \}$ is Borel.
    $\mathcal{H}_{x^\smallfrown n}$ is the set of $\Gamma \in \mathcal{H}_x$ such that $U_x$ splits into at least $n+1$ connected components when we remove vertex $k-1$.
    That is, $\G \in \mathcal{H}_{x^\smallfrown n}$ if and only if $\G \in \mathcal{H}_x$ and there is an $n+1$-tuple of vertices $I = (x_1,...,x_n+1)$ such that each $x_i \in U_x$, and the $x_i$ are pairwise disconnected in $\Gamma[>k-1]$.
    Thus, $\mathcal{H}_{x\smallfrown n}$ is Borel.

    Finally, we need to show that for any $I \subseteq \N$, $\mathcal{H}_{x\smallfrown n, I}$ is Borel.
    Let $j$ be the smallest vertex in $I$.
    $\G \in \mathcal{H}_{x\smallfrown n, I}$ if and only if
    \begin{itemize}
        \item $\G \in \mathcal{H}_{x^\smallfrown i}$, 
        \item for some $N$-tuple of vertices $J = \{x_1,...,x_N\}$ with each $x_i < j$, $\G \in \mathcal{H}_{x,I\cup J}$ and $n \le N < \infty$,
        \item for all $N+1$-tuples $J'=\{y_1,...,y_{N+1}\}$, with each $x_i < j$, $\G \notin \mathcal{H}_{x,I\cup J}$,
        \item for some $n$-tuple of vertices $K = \{ z_1,...,z_n \} \subseteq J$ with each $z_i < j$, the $z_i$ and $j$ are pairwise disconnected in $\Gamma[>k-1]$, and
        \item for each $x_i \in J$, $x_i$ is connected to one of the $z_k$.
    \end{itemize}
     Therefore, $\mathcal{H}_{x^\smallfrown n, I}$ is Borel, as desired.
     That $\G \mapsto \tau_\Gamma^\ell$ follows from the fact that checking whether $U_x$ has positive rank is Borel.
\end{proof}

\begin{COR}
    The end space maps $E_\star, E_{\ell,\star} : \MG_\star \to \mathcal{K}(\N^\N)$ are Borel.
\end{COR}

\bibliographystyle{amsalpha} 
\bibliography{ref}

\end{document}